\documentclass[10pt,a4paper,reqno]{amsart}
\usepackage{mathrsfs}
\usepackage{calc}
\usepackage{cite}
\usepackage{amssymb,amsmath,latexsym,amsthm,enumerate,amsfonts,graphicx,times}
\usepackage[mathscr]{euscript}
\usepackage{amssymb}
\usepackage{amsmath}

\def\dj{d\kern-0.4em\char"16\kern-0.1em}
\def\Dj{\mbox{\raise0.3ex\hbox{-}\kern-0.4em D}}

\def\be{\begin{equation}}
\def\ee{\end{equation}}
\def\bena{\begin{eqnarray*}}
\def\ena{\end{eqnarray*}}

\def\t{\tau}
\def\s{\sigma}

\def\suml{\sum\limits}

\def\dss{\displaystyle}

 \def\E{\mathcal{E}}
 \def\Rd{\mathbf{R}^d}
 \def\Z{\mathbf{Z}_+}

\def\N{\mathbf{N}}

\def\lf {\lfloor}
\def\rf{\rfloor}

\numberwithin{equation}{section}

\newtheorem{te}{Theorem}[section]
\newtheorem{lema}{Lemma}[section]
\newtheorem{prop}{Proposition}[section]
\newtheorem{cor}{Corollary}[section]

\theoremstyle{definition}
\newtheorem{ex}{Example}[section]

\theoremstyle{remark}
\newtheorem{rem}{Remark}[section]

\title{\textbf{Extended Gevrey regularity via weighted matrices}}

\author{Nenad Teofanov}

\address{Department of Mathematics and Informatics,
University of Novi Sad, Novi Sad, Serbia}

\email{nenad.teofanov@dmi.uns.ac.rs}

\author{Filip Tomi\'c}

\address{Faculty of Technical Sciences,
University of Novi Sad, Novi Sad, Serbia}

\email{filip.tomic@uns.ac.rs}



\begin{document}
\begin{abstract}
The main aim of this paper is to compare two recent approaches for investigating the interspace between the union of Gevrey spaces $\mathcal G_t (U)$ and the space of smooth functions $C^{\infty}(U)$. The first approach in the style of Komatsu is based on the properties of two parameter sequences $M_p=p^{\tau p^{\sigma}}$, $\tau>0$, $\sigma>1$. The other one uses weight matrices defined by certain weight functions. We prove the equivalence of the corresponding spaces in the Beurling case
by taking projective limits with respect to matrix parameters, while in the Roumieu case we need to consider a larger space then the one obtained as the inductive limit of extended Gevrey classes.
\end{abstract}

\maketitle

\par

\section{Introduction}

Classes of ultradifferentable functions (also known as smooth functions of ultra-rapid decay)
are usually studied in the framework of one of the two most widely used approaches.
The first one is based on the properties of the defining sequences $M_p$, $p\in \N$, which control the derivatives of the functions, \cite{Komatsuultra1}. For the same purpose, the other approach uses  weights with the certain asymptotical properties, \cite{BMT, BMT2}.
The relation between these weights and the so-called associated functions to  $M_p$ sequences provides a way to compare the theories of  ultradifferentable functions and their dual spaces of ultradistributions.  In many situations these approaches are equivalent. For example, it is  proved in \cite{BMT2}
that the corresponding classes of functions are equal
if the sequence $M_p$ satisfies Komatsu's conditions $(M.1)$, $(M.2)$ and $(M.3)$, see Section \ref{secPreliminaries}. These conditions are relaxed in \cite{Bonet} where $(M.3)$ is replaced by
\be
\label{NoviUslov}
(\exists Q\in \N)\quad\liminf_{p\to\infty}\frac{m_{Qp}}{m_p}>1,
\ee
with $\dss m_p=M_p/M_{p-1}$.

In this paper we study the equivalence of the approaches by considering  specific sequences which do not satisfy $(M.2)$. To that end we exploit the powerful technique based on weight matrices introduced in \cite{RS}. Broadly speaking, weight matrices are families of sequences. For instance, $\{p!^t\}_{t>0}$ is a weight matrix that consists of Gevrey sequences. More generally, for a given weight function $\omega$ (see Subsection \ref{secNotacija} for the definition)
one can observe matrices of the form $\dss {\mathcal M}=\{M_p^H\}_{H>0}$ where
\be
\label{MatrixW}
\dss M_p^H=e^{\frac{1}{H}\varphi^*(Hp)},\quad p\in \N,
\ee
and $\varphi^*$ is the Young conjugate of $\dss \varphi(t)=\omega(e^t)$, see \eqref{Young}.
This approach allows to prove that the corresponding classes of functions
are equivalent in certain situations even if $(M.2)$ is violated, see \cite{RS, Rainer}.

We consider the two parameter defining sequences of the form $M^{\t,\s}_p=p^{\t p^{\s}}$, $\t>0$, $\s>1$, cf. \cite{PTT-01}. Such sequences do not satisfy $(M.2)$ for any choice of parameters $\t>0$ and $\s>1$, hence the corresponding classes of functions $\E_{\t,\s}(U)$ (extended Gevrey classes) are not ultradifferentiable. However, related ultradifferentiable classes can be obtained by taking their unions and intersections (inductive and projective limits) with respect to the parameter $\t$ (this follows from Proposition \ref{OsobineKlasa} {\em iv)}).

Extended Gevrey regularity turned out to be convenient when describing certain aspects of some
hyperbolic PDE's. In particular, $\E_{1,2}(U)$ appears in the study of local solvability of strictly hyperbolic PDE's, for which the initial value problem is ill posed in the Gevrey settings (see \cite{CL}). In addition, sequences $M^{\t,\s}_p$ for  $1<\s\leq 2$ are recently used in \cite{Javier} to study the surjectivity of Borel maps for ultraholomorphic classes. For more details concerning $M^{\t,\s}_p$ and  $\E_{\t,\s}(U)$ we refer to \cite{PTT-01, PTT-04, PTT-05}.

In this paper we prove that the extended Gevrey classes are the special case of classes investigated \cite{Rainer, RS} only when considering projective and inductive limits with respect to the (matrix) parameter $\t$. More precisely, in the Beurling case we prove the equality of the corresponding spaces, while in the Roumieu case the equivalence holds when  the corresponding inductive limit is replaced by a larger space of test functions (see \eqref{OdnosSigma} and \eqref{GlavnaTeoremaJednakost}).

We start by proving that the function $T_{\t,\s,h}(k)$ associated to the sequence $M^{\t,\s}_p=p^{\t p^{\s}}$ is equivalent to a weight function in the sense of \cite{Bonet} (see Theorem \ref{GlavnaTeorema}).  For that purpose we need to estimate $T_{\t,\s,h}(k)$. This is done in \cite[Theorem 2.1]{PTT-04} by using the properties of the Lambert $W$ function. In Proposition \ref{PhiSigma} (see also Lemma \ref{OsobineAsocirane}) we use another  technique to obtain similar estimates.
Consequently, we conclude that  $\dss \{M^{\t,\s}_p\}_{\t>0}$ and $\dss \{e^{\frac{1}{H}\varphi_{\s}^*(Hp)}\}_{H>0}$ are equivalent matrices for a suitable function $\varphi_{\s}$, which implies that the corresponding classes of functions
given by these matrices coincide.

Although Theorem \ref{GlavnaTeorema}, as the main result of the paper, connects the approach from \cite{PTT-01, PTT-04, PTT-05} with the one given in \cite{RS, Rainer}, let us mention an important difference between them.
In contrast to the usual Carleman classes and corresponding part in \cite{RS, Rainer}, in the norm \eqref{Norma} we consider $h^{|\alpha|^\s}$, $\s>1$, in denominator. Thus the parameter $\s$ plays an important role in our construction which can not be revealed by using the techniques from
\cite{RS, Rainer}. For example, the spaces $\E_{\t,\s}(U)$ are closed under finite order differentiation for any choice of parameters $\t>0$ and $\s>1$.
In addition, the parameters $h$ and $\s$ provide a "fine tuning" in the gap between the union of Gevrey spaces and $C^{\infty}$ (see  Proposition \ref{OsobineKlasa} {\em i)}).

%Finally, the growth of $M^{\t,\s}_p=p^{\t p^{\s}}$ is controlled by the parameter $\s$, see $\widetilde{(M.5)}$. Hence, we may say that in this paper we are dealing with families of weight functions (depending on $\s$) instead of just fixing only one weight function as it is done in
%\cite{RS, Rainer}.

We end this introductory section with a review of some basic notions.

\subsection{Basic notions and notation}\label{secNotacija}

We will use the standat notation ${\bf N}$, $\Z$, ${\bf R}$,  ${\bf R}_+$, ${\bf C}$, for the sets of nonnegative
integers, positive integers, real numbers, positive real numbers and complex numbers, respectively.
%For $x \in \Rd$
%we put  $\langle x \rangle=(1+|x|^2)^{1/2}$.
The floor function of  $x\in {\bf R}_+$ is denoted by $\lf x \rf:=\max\{m\in \N\,:\,m\leq x\}$.
For a multi-index
$\alpha=(\alpha_1,\dots,\alpha_d)\in {\bf N}^d$ we write
$\partial^{\alpha}=\partial^{\alpha_1}\dots\partial^{\alpha_d}$ and
$|\alpha|=|\alpha_1|+\dots |\alpha_d|$. By $\# A$ we denote the number of elements of the finite set $A$. We write $\ln_+ x=\max\{0,\ln x\}$, $x>0$. %Throughout the paper $\t>0$ and $\s>1$.

An essential role in our analysis is played by the \emph{Lambert $W$ function}, which is defined as the inverse of $z e^{z}$, $z\in {\bf C}$. By $W(x)$, $x\geq 0$, we denote the restriction of its principal branch, and we review some of its basic properties as follows:
\begin{itemize}
\item[$(W1) \quad$] $W(0)=0$, $W(e)=1$, $W(x)$ is continuous, increasing and concave on $[0,\infty)$,
\vspace{0.2cm}
\item[$(W2) \quad$] $W(x e^{x})=x$ and $ x=W(x)e^{W(x)}$,  $x\geq 0$,
\vspace{0.2cm}
\item[$(W3) \quad $] $\displaystyle\ln x -\ln(\ln x)\leq W(x)\leq \ln x-\frac{1}{2}\ln (\ln x)$,  $ x\geq e$.
\end{itemize}

Note that $(W2)$ implies
\be
\label{PosledicaLambert1}
W(x\ln x)=\ln x,\quad x>1.
\ee
By using $(W3)$ we obtain
\be
\label{PosledicaLambert1.5}
W(x)\sim \ln x, \quad x\to \infty,
\ee
and therefore
\be
\label{PosledicaLambert2}
W(C x)\sim W(x),\quad x\to \infty,
\ee
for any $C>0$. We refer to \cite{LambF} for more details concerning the Lambert function.

Recall (see \cite{Bonet}), a non-negative, continuous, even and increasing function $\omega$ defined on $\mathbf R$, $\omega(0)=0$, is called \emph{weight function} if it satisfies the following conditions:
\begin{itemize}
\item[($\alpha$)] $\omega(2t)=O(\omega(t)),\quad t\to \infty,$
\vspace{0.1cm}
\item[($\beta$)] $ \omega(t)=O(t),\quad t\to\infty$
\vspace{0.1cm}
\item[($\gamma$)] $o(\omega(t))=\log t,\quad t \to \infty,$
\vspace{0.1cm}
\item [($\delta$)] $\varphi(t)=w(e^t),\quad {\rm is\, convex}.$
\end{itemize}

Young's conjugate of the function $\varphi$ (defined as above) is given by
\be
\label{Young}
 \dss \varphi^*(k)=\sup_{t>0}(kt - \varphi(t)),\quad k\geq 0.
\ee
Some classical examples of weight functions are
\be
\label{BMTexamples}
\omega (t) = \ln^{s}_+ |t|,\quad\quad \omega (t)=\frac{|t|}{\ln^{s-1} (e+|t|)},\quad s>1,\, t\in \mathbf R.
\ee
Moreover, $\omega (t)=|t|^s$ is a weight function if and only if $0<s\leq 1$. Note that  by \eqref{PosledicaLambert1.5} it follows that $\omega(t)=W(|t|)$ is not a weight functions since the condition $(\gamma)$ is not satisfied.

Functions $f$ and $g$ are equivalent if $f=O(g)$ and $g=O(f)$, and we will write $f\asymp g$. In particular, if $\omega$ is a weight function and $\omega_1\asymp\omega$ then
\be
\label{OcenaYoung}
A \varphi^*(y / A) \leq \varphi_1^*(y) \leq B\varphi^*(y/B)\quad y>0,
\ee
for some $A,B>0$, where $\varphi (t) = \omega(e^t)$, $\varphi_1 (t) = \omega_1 (e^t)$ and $\varphi^*$, $\varphi_1^*$ are their Young's conjugates, respectively (see \cite{BMT2}).

Throughout the paper we assume that $\t>0$ and $\s>1$, unless stated  otherwise.

\section{Preliminaries}
\label{secPreliminaries}

In this section we recall the definitions of weight functions, weight sequences and their associated functions, and classes of ultradifferentiable functions related to the extended Gevrey regularity. We also list their main properties that will be used in Section \ref{secGlavna}. We proceed with weight sequences introduced in \cite{PTT-01}.
\subsection{Weight sequences}
\label{secNizovi}

In this subsection we consider sequences of the form $M_p^{\tau,\s}=p^{\t p^\s}$,  $M^{\t,\s}_0=1$, $\tau>0$, $\s>1$.  Since $\dss (M^{\t,\s}_p)^{1/p}\to \infty$ when $p\to\infty$ such sequences are examples of weight sequences considered in \cite{Rainer}.

Note that

\be
\label{PrvaNejednakost}
M^{\t_1,\s_1}_{p}\leq M^{\t_2,\s_2}_p,\quad 0< \t_1\leq \t_2,\quad 1<\s_1\leq \s_2,\quad p\in \N.
\ee

Moreover, $M_p^{\tau,\s}=p^{\t p^\s}$, $\tau>0$, $\s>1$ ($M^{\t,\s}_0=1$),
satisfies the following conditions  (see \cite{PTT-01} for the proof):

\vspace{2mm}
$(M.1)$ $\dss ({M_p^{\t,\s}})^2\leq {M_{p-1}^{\t,\s}}{M_{p+1}^{\t,\s}}$,\,\,$p\in \Z$,\\
\vspace{1mm}

$\widetilde{(M.2)'}$  $(\exists C>0)\quad M_{p+1}^{\t,\s}\leq C^{p^{\s}}M_p^{\t,\s}$, $p\in \N$,\\

$\widetilde{(M.2)}$ $ (\exists C>0)$\quad $M_{p+q}^{\t,\s}\leq C^{p^{\s}+q^{\s}}M_p^{2^{\s-1}\t,\s}M_q^{2^{\s-1}\t,\s},\quad p,q\in \N,$ \\

$(M.3)'$ $  \displaystyle
\suml_{p=1}^{\infty}\frac{M_{p-1}^{\t,\s}}{M_{p}^{\t,\s}} <\infty,
$
\vspace{1mm}

$\widetilde{(M.4)}$ $\dss (\forall h>0)\, (\exists C>0)\quad  M^{\t_1,\s}_{p}\leq C\,h^{p^{\s}} M^{\t_2,\s}_p,\quad 0<\t_1<\t_2,\quad \s>1$,
\vspace{1mm}

$\widetilde{(M.5)}$ $\dss (\forall h>0)\, (\exists C>0)\quad M^{\t_1,\s_1}_{p}\leq C h^{p^{\s_2}} M^{\t_2,\s_2}_p,\,\, \t_1,\t_2>0,\, 1<\s_1<\s_2.$

Note that  $\widetilde{(M.4)}$ implies
$$C h^{p^{\s}} M^{\t,\s}_p\geq M^{\t/2,\s}_p \geq 1,\quad C, h>0,\quad p\in \N,$$ and hence we obtain weaker condition

\vspace{2mm}

$\widetilde{(M.4)’}$  $(\forall h>0)\, (\exists C>0)\quad{ h^{p^\s} M_p^{\t,\s}}\geq C   ,\quad p\in \N$.

\begin{rem}
Let us briefly comment the case $\s=1$. Then the conditions $\widetilde{(M.2)'}$ and $\widetilde{(M.2)}$ are classical Komtasu's  ${(M.2)'}$ and ${(M.2)}$ (respectively) for the Gevrey sequence $M_p=p!^\t$. Moreover, $\widetilde{(M.4)}$ also holds. The theory of Gevrey functions is a classical one (see \cite{KomatsuNotes, Rodino} and references therein), hence we are interested in the case $\s>1$.

Note that $\widetilde{(M.5)}$ is also true for the case $\s_2>\s_1=1$ (see \cite{PTT-01}).
\end{rem}

\vspace{1mm}
Recall (see \cite{Rainer}), a family of weight sequences ${\mathcal M}$ is called \emph{weight matrix} if
\be
\label{weightmatrix}
(\forall M_p, N_p \in{\mathcal M})\quad M_p\leq N_p\,\,\vee\,\,  N_p\leq M_p,\quad p\in \N.
\ee

\begin{ex}
For fixed $\s>1$ and $\s_2=\s_1=\s$, \eqref{PrvaNejednakost} implies that ${\mathcal M}_{\s}=\{M^{\t,\s}_p\}_{\t>0}$ is a weight matrix. Similarly,  ${\mathcal M}_{\t}=\{M^{\t,\s}_p\}_{\s>1}$ is a weight matrix for any given $\t > 0$.
Nevertheless, if we observe ${\mathcal M}=\{M^{\t,\s}_p\}_{\t>0,\s>1}$ then
for $\t_1>\t_2$ and $\s_1<\s_2$ we can only prove that
$$M^{\t_1,\s_1}_p\leq C M^{\t_2,\s_2}_p,\quad p\in \N,$$ for a large positive constant $C$ (see $\widetilde{(M.5)}$). Thus ${\mathcal M}=\{M^{\t,\s}_p\}_{\t>0,\s>1}$ does not satisfy \eqref{weightmatrix}.
\end{ex}

For two weight matrices ${\mathcal M}$ and ${\mathcal N}$ we write ${\mathcal M}\lesssim{\mathcal N}$ if
$$(\forall M_p\in {\mathcal M})\,(\exists N_p\in {\mathcal N})\,(\exists C>0) \quad  M_p\leq C^p N_p,\quad  p\in\N.$$
We say that ${\mathcal M}$ and ${\mathcal N}$ are \emph{equivalent} if ${\mathcal M}\lesssim{\mathcal N}$ and ${\mathcal N}\lesssim{\mathcal M}$  (see \cite{Rainer}).

\begin{rem}
\label{RemarkEkv}
Let $\omega$ be a weight function and $\omega_1$ equivalent to $\omega$. Notice that  $\dss \{M_p^H=e^{\frac{1}{H}\varphi^*(Hp)}\}_{H>0}$ and $\dss \{M_p^{H_1}=e^{\frac{1}{H_1}\varphi_1^*(H_1p)}\}_{H_1>0}$ are equivalent matrices due to
inequalities in \eqref{OcenaYoung}. In particular, when investigating matrices of the form \eqref{MatrixW} it is sufficient to consider functions that are equivalent to weights.
\end{rem}

Put
\be
\label{malomp}
\dss m^{\t,\s}_p=\frac{M_p^{\t,\s}}{M_{p-1}^{\t,\s}}, \qquad p\in \Z.
\ee
By $(M.1)$ it follows that $m^{\t,\s}_p$ is an increasing sequence.
Moreover, the following Lemma holds.

\begin{lema}
\label{Lemamp}
Let  $M_p^{\tau,\s}=p^{\t p^\s}$,  $M^{\t,\s}_0=1$, $\tau>0$, $\s>1$, and let $m^{\t,\s}_p$ be given by
\eqref{malomp}. Then the following inequalities hold
\be
\label{Nejednakost_mp}
\Big(\frac{e}{2^\s}\Big)^{\frac{\t p^{\s-1}}{2^{\s-1}}}p^{\frac{\t \s p^{\s-1}}{2^{\s-1}}}\leq m^{\t,\s}_p \leq e^{\t p^{\s-1}}p^{\t \s p^{\s-1}}, \quad p\geq 2.
\ee
\end{lema}

\begin{proof}

Set $\dss f_{\t,\s}(x)=\t x^\s \ln x$, $x>0$. By the mean value theorem, for every $p\in \Z$ there exists $\theta_p$ such that

\be
\label{MeanValue}
f_{\t,\s}(p) -f_{\t,\s}(p-1)=\t \theta_p^{\s-1}\ln(e \theta_p^{\s}),\quad p-1<\theta_p <p.
\ee For $p\geq 2 \iff p/2\leq p-1$, we obtain

$$\frac{\t p^{\s-1}}{2^{\s-1}}\ln\frac{e p^{\s}}{2^{\s}}\leq \t (p-1)^{\s-1}\ln(e (p-1)^{\s})<\t \theta_p^{\s-1}\ln(e \theta_p^{\s}) <\t p^{\s-1}\ln(e p^{\s}),$$ and by \eqref{MeanValue} we conclude

$$\frac{\t p^{\s-1}}{2^{\s-1}}\ln\frac{e p^{\s}}{2^{\s}}\leq \t p^{\s }\ln p -\t (p-1)^\s\ln (p-1)\leq \t p^{\s-1}\ln(e p^{\s}),\quad p\geq 2. $$ Then \eqref{Nejednakost_mp} follows after taking exponentials.
\end{proof}

\begin{rem}
\label{remM2'}
Note that $\widetilde{(M.2)'}$ follows from the right-hand side of \eqref{Nejednakost_mp}. In particular,

$$M^{\t,\s}_p\leq  e^{\t p^{\s-1}}p^{\t \s p^{\s-1}} M^{\t,\s}_{p-1}\leq C^{p^{\s}} M^{\t,\s}_{p-1} ,\quad p\in \Z,$$ for suitable $C>0$.
\end{rem}

\subsection{Associated function}\label{secAsocirana}
\par
In this subsection we recall the definition and some elementary properties of $T_{\t,\s,h}(k)$, $h>0$,
the associated function to the sequence $M^{\t,\s}_p=p^{\t p^\s}$ given by
\be \label{asociranaProduzena}
\dss T_{\t,\s,h}(k)=\sup_{p\in \N}\ln_+\frac{h^{p^{\s}}k^{p}}{M_p^{\t,\s}}, \;\;\; k>0.
\ee
We refer to \cite{PTT-04} for more details on $T_{\t,\s,h}(k)$.
One of the aims of this paper is to prove that $\omega(k)=T_{\t,\s,h}(|k|)$ is equivalent to a weight function, see Theorem \ref{GlavnaTeorema} {\em i)}.

\begin{rem}
Consider $1<\s \leq 2$. Then by $\widetilde{(M.4)'}$ and Example 20 from \cite{Bonet} we obtain
$$T_{\t,\s,h}(k)\leq A\sup_{p\in \N}\ln_+\frac{k^p}{e^{p^\s}}+B\leq A_1\ln_+^{\frac{\s}{\s-1}} k +B_1 \,\quad k>0,$$ for suitable $A_1>0$ and $B_1\in \mathbf R$ (depending on $\t,\s,h$).  Hence we conclude that $\dss T_{\t,\s,h}(k)$ is dominated by a weight function (see \eqref{BMTexamples}). However this fact does not imply that $T_{\t,\s,h}(|k|)$ is equivalent to a weight function. We will provide additional arguments in the proof of Theorem \ref{GlavnaTeorema}.
\end{rem}

Sharp estimates for $ T_{\t,\s,h}(k)$ are given in \cite{PTT-04}, whenre it is proved that for some $A_1,A_2>0$ and $B_1, B_2\in \mathbf R$ (depending on $\t,\s,h$) the following estimates hold:
\begin{multline}
\label{nejednakostzaTeoremu1}
A_1  {W^{-\frac{1}{\s-1}}({{\mathfrak R}(h,k)})}\,{\ln_+}^{\frac{\s}{\s-1}}k +B_1 \leq
T_{\t,\s,h}(k)\\ \leq
 A_2 {W^{-\frac{1}{\s-1}}({{\mathfrak R}(h,k)})}\,{\ln_+}^{\frac{\s}{\s-1}}k +B_2,
\end{multline}
where
$$\dss {\mathfrak R}(h,k):=h^{-\frac{\s-1}{\t}}e^{\frac{\s-1}{\s}}\frac{\s-1}{\t \s}\ln(e+k),\quad h,k>0,$$ and $W$ is the principal branch of the Lambert function. Note that these estimates hold for any choice of parameters $h>0$, $\t>0$ and $\s>1$.

We write $T_{\t,\s}(k):=T_{\t,\s,1}(k)$ for the classical associated function associated to $M^{\t,\s}_p$ (see \cite{Komatsuultra1}).  We end this subsection with a simple result which will be used in the sequel.

\begin{lema}\label{OsobineAsocirane}
Let $T_{\t,\s,h}(k)$ be given by \eqref{asociranaProduzena}, and let $T_{\t,\s}(k):=T_{\t,\s,1}(k)$.
Then for any given  $h>0$ and $\t_2>\t>\t_1>0$ there exists $A,B\in {\mathbf R}$ such that
$$  T_{\t_2,\s}(k)+A \leq T_{\t,\s,h}(k)\leq  T_{\t_1,\s}(k)+B,\quad k>0.$$
\end{lema}

\begin{proof}
Note that $\widetilde{(M.4)}$ implies
$$C_2 \frac{k^p}{M^{\t_2,\s}_p}\leq \frac{\,h^{p^\s}k^p}{M^{\t,\s}_p}\leq C_1  \frac{ k^p}{ M^{\t_1,\s}_p},\quad k>0,\quad p\in \N, $$
and the conclusion follows after taking logarithms and the supremum with respect to $p\in \N$.
\end{proof}

\subsection{Extended Gevrey classes}\label{secKlase}

In this subsection we recall the definition of extended Gevrey classes and some of their basic properties. We consider general sequences which satisfy conditions $(M.1) - \widetilde{(M.5)}$.

Let $U$  be an open set in $\Rd$ and $K\subset \subset U$ be a regular compact set.  We denote by $\dss {\E}_{\t,\s,h}(K)$ the Banach space of functions $\phi \in  C^{\infty}(K)$ such that
\begin{equation} \label{Norma}
\| \phi \|_{{\E}_{\t,\s,h}(K)}=\sup_{\alpha \in \N^d}\sup_{x\in K}
\frac{|\partial^{\alpha} \phi (x)|}{h^{|\alpha|^{\s}}  M_{|\alpha|} ^{\t,\s} }<\infty.\,
\end{equation}
Note that
$$ \displaystyle
{\E}_{\t_1, \s_1, h_1}(K)\hookrightarrow {\E}_{\t_2,\s_2,h_2}(K), \;\;\;
0<h_1<h_2, \; 0<\t_1<\t_2, \; 1<\s_1<\s_2,
$$
where $\hookrightarrow$ denotes the strict and dense inclusion. We define spaces of Roumieu and Beurling type
by introducing the following inductive and projective limit topologies (respectively)
\begin{equation}
\label{NewClassesER} {\E}_{\{\t,\s\}}(U)=\varprojlim_{K\subset\subset U}\varinjlim_{h\to
\infty}{\E}_{\t,\s,h}(K),
\end{equation}

\begin{equation}
\label{NewClassesEB}   {\E}_{(\t,\s)}(U)=\varprojlim_{K\subset\subset U}\varprojlim_{h\to 0}{\E}_{\t,\s,h}(K).
\end{equation}

We omit the brackets if we consider either  $\{\t,\s\}$ or $(\t,\s)$.

\begin{rem}
The condition  $(M.3)'$ implies that ${\E}_{\t,\s}(U)$ contains compactly supported functions. The construction of smooth compactly supported functions which are not in Gevrey classes but which belong to ${\E}_{\t,\s}(U)$  can be found in \cite{PTT-01}.
\end{rem}

Extended Gevrey classes given by \eqref{NewClassesER} and \eqref{NewClassesEB} are studied in \cite{PTT-01, PTT-04, PTT-05}. For the convenience of the reader, we collect some of their basic properties in the following Proposition. Recall, the Gevrey class of index $t>1$ is given by ${\mathcal G}_t (U)=\E_{\{t,1\}}(U)$, see \eqref{NewClassesER}.

\begin{prop}
\label{OsobineKlasa}
Let $U$  be an open set in $\Rd$. Let ${\E}_{\{\t,\s\}}(U)$ and $ {\E}_{(\t,\s)}(U)$ be given by \eqref{NewClassesER} and \eqref{NewClassesEB} respectively, and let $\varinjlim$ and $\varprojlim$ denote the corresponding inductive and projective limits respectively.
Then  the following is true:
\begin{itemize}
\item[$i)$] For $\s_2>\s_1>1$ we have
\begin{multline}
\varinjlim_{t\to\infty}{\mathcal G}_t (U)\hookrightarrow\varprojlim_{\t\to 0}{\E}_{\{\t,\s_1\}}(U)=\varprojlim_{\t\to 0}{\E}_{(\t,\s_1)}(U)\hookrightarrow\varinjlim_{\t\to \infty}{\E}_{(\t,\s_1)}(U)\\
=\varinjlim_{\t\to \infty}{\E}_{\{\t,\s_1\}}(U)\hookrightarrow \varprojlim_{\t\to 0}{\E}_{\{\t,\s_2\}}(U)\hookrightarrow C^{\infty} (U).\nonumber
\end{multline}
\item[$ii)$] ${\E}_{\t,\s}(U) $ are closed under the pontwise multiplication.
\item[$iii)$] ${\E}_{\t,\s}(U) $ are closed under finite order derivation.
\item[$iv)$] For $a_{\alpha}\in {\E}_{(\t,\s)}(U) $ (resp. $a_{\alpha}\in {\E}_{\{\t,\s\}}(U) $) define
$$P(x,\partial)=\sum_{|\alpha|=0}^{\infty}a_{\alpha}(x)\partial^{\alpha},$$
such that for every $K\subset\subset U$ there exists $L>0$ and for every $h>0$ there exists $A>0$ (resp. for every $K\subset\subset U$ there exists $h>0$ and for every $L>0$ there exists $A>0$) so that
$$\sup_{x\in K}|\partial^{\beta}a_{\alpha}(x)|\leq A h^{|\beta|^{\s}}M^{\t,\s}_{|\beta|}\frac{L^{|\alpha|^\s}}{M^{2^{\s-1}\t,\s}_{|\alpha|}}.
$$
Then $\dss P(x,\partial)\,:\,{\E}_{\t,\s}(U)\to {\E}_{2^{\s-1}\t,\s}(U)  $ is a continuous and linear mapping.
\end{itemize}
\end{prop}

Let $h=1$ in \eqref{Norma}. We introduce the following spaces
\be
\label{Esigma}
\E_{\{\s\}}(U)=\varinjlim_{\t\to \infty}\E_{\t,\s}(U),\quad \E_{(\s)}(U)=\varprojlim_{\t\to 0} \E_{\t,\s}(U),
\ee

\be
\label{EsigmaTilde}
\E_{\{\s\}} ^R (U)= \varprojlim_{K\subset\subset U}\varinjlim_{\t\to
\infty}{\E}_{\t,\s,1}(K),\quad \E_{(\s)} ^B(U)=\varprojlim_{K\subset\subset U}\varprojlim_{\t\to
0}{\E}_{\t,\s,1}(K).
\ee
Note that Proposition \ref{OsobineKlasa} {\em i)}, and the order of quantifiers in the definition of spaces \eqref{Esigma} and \eqref{EsigmaTilde} implies that
\be
\label{OdnosSigma}
\varinjlim_{t\to\infty}{\mathcal G}_t (U)\hookrightarrow \E_{(\s)}(U)=\E_{(\s)} ^B(U)\hookrightarrow \E_{\{\s\}}(U)\hookrightarrow\E_{\{\s\}} ^R(U)\hookrightarrow C^{\infty}(U).
\ee
Notice that, unlike $ \E_{\t,\s}(U) $,  $\E_{\{\s\}}(U)$, $ \E_{(\s)}(U) $, are classes of ultradifferentiable functions. This follows from Proposition \ref{OsobineKlasa} {\em iv)}. Moreover, ultradifferentiability of $\E_{\{\s\}} ^R(U)$ follows from the arguments given in \cite{RS}.

\section{Main result}\label{secGlavna}

In this section we first give an estimate for $T_{\t,\s}(k)=T_{\t,\s,1}(k)$ which is introduced in Subsection \ref{secAsocirana}. Note that $T_{\t,\s}(k)$ satisfies estimates \eqref{nejednakostzaTeoremu1} when $h=1$. Therefore the following Proposition follows directly from \cite[Theorem 2.1]{PTT-04}. However, here we give an independent proof.

\begin{prop}
\label{PhiSigma}
Let $T_{\t,\s}(k)=T_{\t,\s,1}(k)$ be given by \eqref{asociranaProduzena} with $h=1$, and let
$W(t)$, $t> 0$, denote the restriction of the principal branch of the Lambert $W$ function.
If
$\dss\varphi_{\s}(t)=\frac{t^{\frac{\s}{\s-1}} }{W ^{\frac{1}{\s-1}}(t)}$, $t>0$, and $\varphi_{\sigma}(0)=0$, then we have
$$
B_{\t,\s}\varphi_{\s}(\ln_+ k)+\widetilde{B}_{\t,\s}  \leq T_{\t,\s}(k)\leq A_{\t,\s}\varphi_{\s}(\ln_+ k)+\widetilde{A}_{\t,\s},\quad k>0,
$$
for suitable constants $A_{\t,\s},B_{\t,\s}>0$ and $\widetilde{A}_{\t,\s},\widetilde{B}_{\t,\s}\in \mathbf R$.
\end{prop}

\begin{proof}
For $\lambda>0$ we let
$$\dss m_{\t,\s}(\lambda)=\#\{p\in \Z\,|\, m^{\t,\s}_{p}\leq \lambda\},$$
and note that $m_{\t,\s}(\lambda)=0$ for all $0<\lambda< 1$. This is due to the fact that $m_1=1$ and $m^{\t,\s}_p$ is increasing.

Since $M^{\t,\s}_p$ satisfies $(M.1)$ we can write (see \cite{Komatsuultra1, PilipovicKnjiga})

$$T_{\t,\s}(k)=\int_{0}^k \frac{m_{\t,\s}(\lambda)}{\lambda}d\lambda=\int_{1}^k \frac{m_{\t,\s}(\lambda)}{\lambda}d\lambda.$$  In the sequel we estimate $m_{\t,\s}(\lambda)$ when $\lambda \geq 1$.

Put
$$m^C_{\t,\s}(\lambda)=\#\{p\in \Z\,|\, C^{p^{\s-1}} p^{\t p^{\s-1}}\leq \lambda\}, \quad C>0.$$ Then \eqref{Nejednakost_mp} implies that
\be
\label{OcenaMCP}
m^{C_1}_{(\t \s),\s}(\lambda) \leq m_{\t,\s}(\lambda)\leq m^{C_2}_{(\t\s)/2^{\s-1},\s}(\lambda) ,\quad \lambda\geq 1,\nonumber
\ee where $C_1=e^{\t}$ and $C_2=(e/2^{\s})^{\t/2^{\s-1}}$. In particular,

\be
\label{OcenaTIntegral}
\int_{1}^k \frac{ m^{C_1}_{(\t \s),\s}(\lambda)   }{\lambda}d\lambda\leq T_{\t,\s}(k)\leq\int_{1}^k \frac{ m^{C_2}_{(\t\s)/2^{\s-1},\s}(\lambda)}{\lambda}d\lambda, \quad k>0.
\ee

Next we note that

\begin{multline} C^{p^{\s-1}} p^{\t  p^{\s-1}}\leq \lambda \iff C^{\frac{\s-1}{\t}} p^{\s-1} \ln (C^{\frac{\s-1}{\t}} p^{\s-1})\leq C^{\frac{\s-1}{\t}} \frac{\s-1}{\t}\ln\lambda\\
\iff \ln(C^{\frac{\s-1}{\t}} p^{\s-1}) \leq W( C^{\frac{\s-1}{\t}} \frac{\s-1}{\t}\ln\lambda)\iff\\
 p\leq C^{-\frac{1}{\t}}{ e}^{\frac{1}{\s-1}W( C^{\frac{\s-1}{\t}} \frac{\s-1}{\t}\ln\lambda)},\quad C>0,\,\, \lambda\geq 1,\nonumber
\end{multline} where for the second equivalence we used property \eqref{PosledicaLambert1} of the Lambert function.

This calculation shows that
\begin{equation}
\label{ocenaMCPrevisited}
m^C_{\t,\s}(\lambda)=\Big\lfloor C^{-\frac{1}{\t}}{ e}^{\frac{1}{\s-1}W( C^{\frac{\s-1}{\t}} \frac{\s-1}{\t}\ln\lambda)}\Big\rfloor, \quad \lambda\geq 1.\nonumber
\end{equation} and therefore

$$\int_{1}^k \frac{m^C_{\t,\s}(\lambda) }{\lambda}d\lambda \sim C^{-\frac{1}{\t}} \int_{1}^k \frac{{e}^{\frac{1}{\s-1}W( C^{\frac{\s-1}{\t}} \frac{\s-1}{\t}\ln\lambda)}}{\lambda} d\lambda,\,\quad k\to \infty. $$ It remains to compute

$$ I^C_{\t,\s}(k):= C^{-\frac{1}{\t}} \int_{1}^k \frac{{e}^{\frac{1}{\s-1}W( C^{\frac{\s-1}{\t}} \frac{\s-1}{\t}\ln\lambda)}}{\lambda} d\lambda, \quad C>0.$$

Set $C_{\t,\s}= C^{\frac{\s-1}{\t}} \frac{\s-1}{\t}$. Note that after the substitution $t=C_{\t,\s}\ln \lambda$ we obtain
\be
\label{ICocena1}
  I^C_{\t,\s}(k)=C^{-\frac{\s}{\t}}\frac{\t}{\s-1}\int_{0}^{C_{\t,\s} \ln k} {e}^{\frac{1}{\s-1}W(t)} dt.
\ee
Another change of variables $W(t)=s $ $( t=s e^s, dt =(s+1)e^s ds)$, and integration by parts yields
\be
\label{ICocena2}
\int {e}^{\frac{1}{\s-1}W(t)}dt =\int  {e}^{\frac{\s s}{\s-1}} (s+1) ds ={e}^{\frac{\s s}{\s-1}}\frac{\s-1}{\s}\Big(s +\frac{1}{\s}\Big),
\ee where we use indefinite integral just for the notational convenience.

Now using $(W2)$ property of the Lambert function and \eqref{PosledicaLambert2},  by \eqref{ICocena1} and \eqref{ICocena2} we have

\begin{multline}
\label{ocenaICP}
 I^C_{\t,\s}(k)=\frac{\t}{\s} C^{-\frac{\s}{\t}} {e}^{\frac{\s s}{\s-1}} \Big(s +\frac{1}{\s}\Big)\Big|_{0}^{W(t)}=\\
\frac{\t}{\s} C^{-\frac{\s}{\t}} \Big(\frac{ t}{W (t)}\Big)^{\frac{\s}{\s-1}}  (W (t)+\frac{1}{\s})\Big|_{0}^{C_{\t,\s}\ln k}  \asymp \t^{-\frac{1}{\s-1}}\varphi_{\s}(\ln_+  k) + \widetilde{C}_{\t,\s} ,\  k>0,
\end{multline} for some $\widetilde{C}_{\t,\s}\in \mathbf R$, where the hidden constants are depending only on $\s$.

More precisely, using \eqref{OcenaTIntegral} and \eqref{ocenaICP} we conclude

\be
\label{KonacnaocenaAsocirana}
 B_{\s}\t^{-\frac{1}{\s-1}}\varphi_{\s}(\ln_+  k)+\widetilde{B}_{\t,\s} \leq T_{\t,\s}(k)\leq A_{\sigma}\Big(\frac{\t}{2^{\s-1}}\Big)^{-\frac{1}{\s-1}}\varphi_{\s}(\ln_+  k)+\widetilde{A}_{\t,\s},\,\,k>0,
\ee for suitable $A_{\s}, B_{\s}>0$ and $\widetilde{A}_{\t,\s},\widetilde{B}_{\t,\s}\in \mathbf R$.
This completes the proof.
\end{proof}

Following \cite{Rainer, RS} we introduce the Banach space ${\mathcal B}_{H,\s}(K)$, $K\subset\subset U$, with the norm

$$\|\phi \|_{{\mathcal B}_{H,\s}(K)}=\sup_{\alpha \in \N^d}\sup_{x\in K}
{|\partial^{\alpha} \phi (x)|}e^{-\frac{1}{H}\varphi_{\sigma}^{*}(H p)},\quad H>0,$$ where $\varphi^*_{\s}$ is Youngs conjugate of the function $\dss \varphi_\s$ introduced in Proposition \ref{PhiSigma}.

We introduce the corresponding Roumieu and Beurling classes as follows:
$$ {\mathcal B}_{\{\s\}}(U)=\varprojlim_{K\subset\subset U}\varinjlim_{H\to\infty} {\mathcal B}_{H,\s}(U), \quad {\mathcal B}_{(\s)}(U)=\varprojlim_{K\subset\subset U}\varprojlim_{H\to 0} {\mathcal B}_{H,\s}(U), $$
respectively.

Now we can formulate the main result of the paper.

\begin{te} \label{GlavnaTeorema} Fix $\s>1$ and let $\dss\varphi_{\s}(t)$ be as in Proposition \ref{PhiSigma}. Moreover, let $T_{\t,\s,h}$ be given by \eqref{asociranaProduzena}. Then the following is true.
\begin{itemize}
\item[i)] The function $\omega(k)=\varphi_{\s}(\ln_+|k|)$ is  equivalent to a weight function. Moreover, for every $h>0$ and  $\t>0$, the function $\omega(k)=T_{\t,\s,h}(|k|)$ is equivalent to a weight function.
\item[ii)] The weight matrices ${\mathcal M}_{\s}=\{p^{\t p^{\s}} \}_{\t>0}$ and ${\mathcal N}_{\s}=\{\dss e^{\frac{1}{H}\varphi_{\sigma}^{*}(H p)}\}_{H>0}$ are equivalent. In particular,
\be
\label{GlavnaTeoremaJednakost}
{\mathcal B}_{(\s)}(U)= {\mathcal E}_{(\s)} ^B (U),\quad {\mathcal B}_{\{\s\}}(U)=
{\mathcal E}_{\{\s\}} ^R (U)
\ee
where $\dss {\mathcal E}_{(\s)} ^B (U)$ and ${\mathcal E}_{\{\s\}} ^R (U)$ are given in \eqref{EsigmaTilde}.
\end{itemize}
\end{te}

\begin{proof}
\label{BMTrem}
{\em i)} By Proposition \ref{PhiSigma} it follows that functions $T_{\t,\s}(|k|) \asymp \varphi_{\s}(\ln_+|k|)$. Thus, it is sufficient to show that $\dss T_{\t,\s}(|k|)$ is a weight functions (see Remark \ref{RemarkEkv}).

Since $T_{\t,\s}$ is the function associated to $M^{\t,\s}_p$, by \cite[Lemma 12]{Bonet} it is sufficient to show that $m^{\t,\s}_p$ given by \ref{malomp}
satisfies \eqref{NoviUslov}, i.e., that there exists $Q\in \mathbf N$ such that
\be
\label{TrebaPokazatiTeorema}
\liminf_{p\to\infty}\frac{m^{\t,\s}_{Qp}}{m^{\t,\s}_p}>1.
\ee
Note that Lemma \ref{Lemamp} implies
\begin{multline}
\frac{m^{\t,\s}_{3p}}{m^{\t,\s}_p}\geq \frac{\Big(\frac{e}{2^\s}\Big)^{\frac{\t (3p)^{\s-1}}{2^{\s-1}}}(3p)^{\frac{\t \s (3p)^{\s-1}}{2^{\s-1}}}}{e^{\t p^{\s-1}}p^{\t \s p^{\s-1}}}\\
=\Big(\frac{3}{2}\Big)^{\t\s (\frac{3p}{2})^{\s-1}}\exp\Big\{ \t \Big(({3}/{2})^{\s-1}-1\Big) p^{\s-1} \Big\} p^{\t \s \Big(({3}/{2})^{\s-1}-1\Big) p^{\s-1} }\to \infty,\,\, p\to \infty,\nonumber
\end{multline}
and \eqref{TrebaPokazatiTeorema} follows when $Q=3$. In addition, Lemma \ref{OsobineAsocirane} together with Proposition \ref{PhiSigma}  implies that $T_{\t,\s,h}(|k|)\asymp \varphi_{\s}(\ln_+ |k|)$.

{\em ii)} Let $A=A_{\s}>0$ and $B=B_{\s}>0$ be as in \eqref{KonacnaocenaAsocirana}. For $\t>0$ set
\be
\label{Konstante}
H_1= B^{-1} \t^{\frac{1}{\s-1}},\quad H_2 =A^{-1}(\t/2^{\s-1})^{\frac{1}{\s-1}},\nonumber
\ee and note that \eqref{KonacnaocenaAsocirana} implies
$$
C_2\exp\{-\frac{1}{H_2} \varphi_{\s}(\ln_+ k)\}\leq \inf_{p\in \N}\frac{p^{\t p^{\s}}}{k^p}\leq C_1 \exp\{-\frac{1}{H_1} \varphi_{\s}(\ln_+ k)\},\quad k>0.
$$
for suitable constants $C_1,C_2>0$.

In particular, we have

\be
\label{Ocenavarphi}
C_2\exp\{\frac{1}{H_2}(H_2 \,p\ln_+ k- \varphi_{\s}(\ln_+ k)\}\leq p^{\t p^{\s}}\leq C_1 \exp\{\frac{1}{H_1}(H_1\, p\ln_+ k- \varphi_{\s}(\ln_+ k)\},
\ee
 for $p\in \N$ and $k>0$. Putting $t=\ln_+ k$ and taking the supremum with respect to $t\geq 0$, \eqref{Ocenavarphi} implies
\be
\label{OcenaNorme}
C_2 \exp\{\frac{1}{H_2}\varphi_{\s}^*(H_2\,p)\} \leq p^{\t p^{\s}}\leq C_1\exp\{\frac{1}{H_1}\varphi_{\s}^*(H_1\,p)\},\quad p\in \N.
\ee Therefore, matrices  ${\mathcal M}_{\s}$ and ${\mathcal N}_{\s}$ are equivalent.

It remains to prove \eqref{GlavnaTeoremaJednakost}.
We give the proof for the Roumieu case  $\dss {\mathcal B}_{\{\s\}}(U)=
{\mathcal E}_{\{\s\}} ^R (U)$, and omit the proof for  the Beurling case, since it uses similar arguments.

Let $\phi \in {\mathcal B}_{{\s}}(U)$.  Then for arbitrary $K\subset\subset U$ there exists $H>0$ such that $\|\phi \|_{{\mathcal B}_{H,\s}(K)}<\infty$. Putting  $\t=(2AH)^{\s-1}$, \eqref{OcenaNorme} implies
$$ \|\phi \|_{{\mathcal E}_{\t,\s}(K)}  \leq C' \|\phi \|_{{\mathcal B}_{H,\s}(K)},$$ for some $C'>0$.

Conversely, if $\phi\in {\mathcal E}_{\{\s\}} ^R (U)$ then for arbitrary $K\subset\subset U$ there exists $\t>0$ such that $\dss \|\phi \|_{{\mathcal E}_{\t,\s}(K)}<\infty$. Choosing $H=B^{-1} \t^{\frac{1}{\s-1}}$, again by \eqref{OcenaNorme} we have

$$ \|\phi \|_{{\mathcal B}_{H,\s}(K)}  \leq C''\|\phi \|_{{\mathcal E}_{\t,\s}(K)},$$ for suitable $C''>0$. This completes the proof.
\end{proof}

\begin{rem}
Note that $M^{\t,\s}_p=p^{\t p^\s}$ is not a weight sequence in the sense of \cite{Bonet}, since it does not satisfy $(M.2)'$. Instead we use $\widetilde{(M.2)'}$ in a stronger form (see Remark \ref{remM2'}).

Moreover, in the proof of Theorem \ref{GlavnaTeorema} we use the part of \cite[Lemma 12]{Bonet}
for which it is sufficient to assume $(M.1)$ and
$$(M.0)\quad\quad(\exists C>0)\quad M^{\t,\s}_p\geq C p^p   ,\quad p\in \N,$$
which is obviously true (see $\widetilde{(M.4)'}$).
\end{rem}

We conclude the paper with the following Corollary which is an immediate consequence of Theorem \ref{GlavnaTeorema}.

\begin{cor}
For $s>1$ function $\dss\omega(t)=\frac{\ln_+^s|t|}{\ln^{s-1}(\ln (e+|t|))}$, $t\not=0$, $\omega(0)=0$, is equivalent to a weight function.
\end{cor}

%%%%%%%%%%%%%%%%%%%%%%%%%%%%%%%%%%%%%%%%%%
\section{Discussion}

The equivalence of theories of ultradifferentiable functions given by Komatsu's
or the Braun-Meise-Taylor approach are well established in the most classical situations.
Recent approach based on weighted matrices seems to offer a very general construction,
see \cite{RS, Rainer}. In parallel, it is demonstrated in  \cite{PTT-01, PTT-04, PTT-05} that the two-parameter sequences of the form  $M_p=p^{\tau p^{\sigma}}$, $\tau>0$, $\sigma>1$, provide a useful extension of the Gevrey type spaces.

In this paper we show that the projective  limits of extended Gevrey classes
can be viewed as a part of the construction based on the weighted matrices.
The same conclusion holds when the inductive limits of extended Gevrey classes
are replaced by certain slightly larger spaces.
At the same time, extended Gevrey classes  $\E_{\t,\s}(U)$ for fixed $\tau>0$ and $\sigma>1$,
can not be characterized by weight matrices used in \cite{RS, Rainer}
due to the particular role played by the parameter $\sigma$.

While finishing the paper the authors learned about the work in progress
"A comparison of two ways to generalize ultradifferentiable
classes defined by weight sequences" by J. Jim\'enez-Garrido, D. N. Nenning, and G. Schindl,
which is devoted to similar topic considered from a different point of view.
We thank the authors for their fruitful comments on the first version of this paper.

%%%%%%%%%%%%%%%%%%%%%%%%%%%%%%%%%%%%%%%%%%

\section*{Acknowledgement}
This research was funded by Ministry of Education, Science and
Technological Development, Republic of Serbia  Projects no. 451-03-68/2022-14/200125 and {451-03-68/2022-14/200156}.


\vspace*{1cm}


\begin{thebibliography}{999}
% Reference 1
\bibitem{Bonet} Bonet, J., Meise, R., Melikhov. S., \textit{A comparison of two different ways to define classes of ultradifferentiable functions}, Bull. Belg. Math. Soc. Simon Stevin 14 (3) 425 - 444, 2007.

\bibitem{BMT}  Braun, R.W., Meise, R., Taylor, B. A., \textit{Ultra-differentiable functions and Fourier analysis}, Results Math. 17 (3-4), (1990), 206--237.

\bibitem{PilipovicKnjiga} Carmichael, R., Kaminski, A., Pilipovi\'c, S., \textit{Notes on Boundary Values in Ultradistribution Spaces}, Lecture Notes Series of Seul University, 49, 1999.

\bibitem{CL} Cicognani M., Lorentz, D., \textit{Strictly hyperbolic equations with  coefficients low-regular in time and smooth in space}, J. Pseudo-Differ. Oper. Appl., 9, no. 3, 643–675, (2018).

\bibitem{LambF} Corless, R.M. , Gonnet, G.H. , Hare, D.E.G. ,Jeffrey, D.J. ,Knuth, D.E.  ,\textit{On the Lambert W function},
Adv. Comput. Math. 5, (1996), 329--359.


\bibitem{Rainer}  F\" urd\" os, S., Nenning, D.N, Rainer, A, Schindl, Gerhard \textit{Almost analytic extensions of ultradifferentiable functions with applications to microlocal analysis.} J. Math. Anal. Appl. 481 (2020), no. 1, 123451, 51 pp.

%\bibitem{GelfandShilov} Gelfand, I. M. , Shilov, G. E. , Generalized Functions II,  Academic Press, New York, 1968.


\bibitem{Javier} Jim\'enez-Garrido, J., Lastra, A., Sanz, J., \textit{ Extension operators for some ultraholomorphic classes defined by sequnces of rapid growth}, https://doi.org/10.48550/arXiv.2204.01316 (2022).

%\bibitem{H} Lars H\"{o}rmander, \textit{The Analysis of Linear Partial Differential Operators I}, Springer, 1990.

\bibitem{Komatsuultra1} Komatsu, H., \textit{Ultradistributions, I:
Structure theorems and a characterization}. J. Fac. Sci. Univ. Tokyo,
Sect. IA Math., \textbf{20 1} (1973), 25--105.

\bibitem{KomatsuNotes} Komatsu, H., An introduction to the theory of generalized functions, Lecture notes,
Department of Mathematics Science University of Tokyo, 1999.

 \bibitem{BMT2} Meise, R., Taylor, B.A. \emph{ Whitney's extension theorem for ultradifferentiable functions of Beurling type}, Ark. Mat. 26, no. 2, (1988), 265–287.



%\bibitem{GelfandShilov} I. M. Gelfand, G. E. Shilov,  \textit{Generalized Functions II}. Academic Press, New York, 1968.








\bibitem{PTT-01} Pilipovi\'c, S. ,Teofanov, N. , Tomi\'c, F, \textit{On a class of ultradifferentiable functions}, Novi Sad Journal of Mathematics,
45 (1), (2015), 125--142.


%\bibitem{PTT-02} Pilipovi\'c, S. ,Teofanov, N. , Tomi\'c, F. , \textit{Beyond Gevrey regularity}, J. Pseudo-Differ. Oper. Appl., 7, (2016), 113--140.
%\bibitem{PTT-03} Pilipovi\'c, S. ,Teofanov, N. , Tomi\'c, F , \textit{Superposition and propagation of singularities for extended Gevrey regularity}, 2763--2782,  32 (8),  2018.

\bibitem{PTT-04} Pilipovi\'c, S. ,Teofanov, N. , Tomi\'c, F, \textit{A Paley–Wiener theorem in extended Gevrey regularity}, J. Pseudo-Differ. Oper. Appl., 11 , no. 2, 593–612., 2020.

\bibitem{PTT-05} Pilipovi\'c, S. ,Teofanov, N. , Tomi\'c, F, \textit{Boundary values in ultradistribution spaces related to extended Gevrey regularity} , Mathematics , 9(1), 7, 2021.

%\bibitem{Pilipovic-Toft} S. Pilipovi\'c, J. Toft, \textit{Wave-front sets related to quasi-analytic Gevrey sequences}. Preprint availible online at http://arxiv.org/abs/1210.7741v3, (2015).

%\bibitem{P-SRW} K. Pravda-Starov,  L. Rodino, P. Wahlberg,
%\textit{Propagation of Gabor singularities for
%Schr\"odinger equations with quadratic Hamiltonians},
%arXiv:1411.0251v5 [math.AP], (2015).


\bibitem{RS} Rainer, A., Schindl, G.
\textit{Composition in ultradifferentiable classes},
 Studia Math., \textbf{224} 2, 97 -- 131  (2014)



\bibitem{Rodino} Rodino, L. , \textit{Linear Partial Differential Operators in Gevrey Spaces}, World Scientific, 1993.


%\bibitem{TT0} Teofanov, N. , Tomi\'c, F.,
%\textit{Inverse closedness and singular support in extended Gevrey regularity},
%J. Pseudo-Differ. Oper. Appl. 8 (3) (2017), 411--421.

%\bibitem{TT} Teofanov, N. , Tomi\'c, F.
%\textit{Ultradifferentiable functions of class $M^{\t,\s} _p$ and microlocal regularity},
%Generalized functions and Fourier analysis, Advances in Partial Differential Equations, Birkh\"auser, (2017), 193--213.

%\bibitem{TT1} Teofanov, N., Tomić, F., \textit{Extended Gevrey regularity via the short-time Fourier transform}, Advances in microlocal and time-frequency analysis, 455–474, Appl. Numer. Harmon. Anal., Birkhäuser/Springer (2020)

%\bibitem{FT} Tomi\'c, F., A microlocal property of PDOs in $\E_{(\t,\s)}(U)$, Proceedings of The Second Conference on Mathematics in Engineering: Theory and Applications, Novi Sad, (2017).

\end{thebibliography}
\end{document}